\documentclass{amsart}% option [reqno] puts equat numb to the right

\usepackage[all]{xy}
\usepackage[mathscr]{eucal}% So-called Euler fonts, allows \mathscr
\usepackage{amssymb}% allows special arrows like >-> e.g.
\usepackage{amsthm}% allows theorem* , e.g.
\usepackage{bbold}% allows \unit \mathbb{1}
\usepackage{enumerate}% allows easy change of label
\usepackage{array}% allows array
\usepackage{amsmath}% allows pmatrix
\usepackage{hyperref}% 

\numberwithin{equation}{section}% makes equat numb contain the section
\swapnumbers

\newtheorem{Thm}[equation]{Theorem}
\newtheorem*{Thm*}{Theorem}
\newtheorem{Prop}[equation]{Proposition}
\newtheorem{Lem}[equation]{Lemma}
\newtheorem{Cor}[equation]{Corollary}

\theoremstyle{remark}
\newtheorem{Def}[equation]{Definition}
\newtheorem{Not}[equation]{Notation}
\newtheorem{Exa}[equation]{Example}
\newtheorem{Cons}[equation]{Construction}
\newtheorem{Conv}[equation]{Convention}
\newtheorem{Rem}[equation]{Remark}

\newcommand{\nc}{\newcommand}
\nc{\dmo}{\DeclareMathOperator}

\dmo{\Abelem}{Abelem}
\dmo{\Id}{Id}
\dmo{\Spc}{Spc}
\dmo{\chara}{char}%
\dmo{\colim}{colim}
\dmo{\cone}{cone}
\dmo{\Der}{D}% ground notation for derived categories
\dmo{\Ext}{Ext}
\dmo{\rmH}{H}
\dmo{\Hom}{Hom}
\dmo{\id}{id}
\dmo{\Img}{Im}
\dmo{\Mod}{Mod}% sheaves of modules
\dmo{\modname}{mod}%
\dmo{\Mor}{Mor}%
\dmo{\Or}{Or}
\dmo{\pr}{pr}
\dmo{\Proj}{Proj}
\dmo{\Res}{Res}
\dmo{\smallb}{b}% ground exponent for ``bounded''
\dmo{\smallperf}{perf}% ground exponent for ``perfect''
\dmo{\Spec}{Spec}
\dmo{\stab}{stab}% stable category of fin. gen. mod.
\dmo{\supp}{supp}
\dmo{\VB}{VB}% vector bundles

\nc{\AK}{A\MModcat{K}}% most used
\nc{\DbG}{\Db(\kk G)}%most used
\nc{\Db}{\Der^{\smallb}}% derived bounded category
\nc{\Displ}{\displaystyle}
\nc{\Dperf}{\Der^{\smallperf}}% derived category of perfect compl
\nc{\Gsets}{G\mathsf{-sets}}
\nc{\HGK}{\doublequot HGK}% most used
\nc{\Homcat}[1]{\Hom_{\cat #1}}
\nc{\MModcat}[1]{\MMod_{\cat #1}}%
\nc{\MMod}{\,\text{-}\Mod}%
\nc{\Mid}{\,\big|\,}
\nc{\PZG}{\cat C_{\bbZ}(\bbZ G)}% G-lattices
\nc{\SET}[2]{\big\{\,#1\Mid#2\,\big\}}
\nc{\SpcAK}{\Spc(A\MModcat{K})}% most used
\nc{\SpcK}{\Spc(\cat K)}% most used
\nc{\adhpt}[1]{\adh{\{#1\}}}% adherence of a pt
\nc{\adh}[1]{\overline{#1}}% adherence
\nc{\adjto}{\rightleftarrows}
\nc{\bbC}{\mathbb{C}}
\nc{\bbF}{\mathbb{F}}
\nc{\bbR}{\mathbb{R}}
\nc{\bbZ}{\mathbb{Z}}
\nc{\cV}{\mathcal{V}}% for support varieties
\nc{\cat}[1]{\mathscr{#1}}%or: \nc{\cat}[1]{\mathcal{#1}}
\nc{\doublequot}[3]{#1\backslash #2/#3}% double cosets
\nc{\eg}{{e.\,g.}}
\nc{\gp}{\mathfrak{p}}% prime p
\nc{\gq}{\mathfrak{q}}% prime q
\nc{\hook}{\hookrightarrow}
\nc{\ideal}[1]{\langle #1\rangle}
\nc{\ie}{{i.e.}\ }
\nc{\inv}{^{-1}}
\nc{\isotoo}{\mathop{\buildrel \sim\over\too}}
\nc{\isoto}{\buildrel \sim\over\to}
\nc{\kk}{\Bbbk}
\nc{\mmod}{\,\text{--}\modname}%
\nc{\onto}{\mathop{\twoheadrightarrow}}
\nc{\op}{{^{\operatorname{op}}}}
\nc{\otoo}[1]{\overset{#1}{\,\too\,}}
\nc{\oto}[1]{\overset{#1}\to}
\nc{\ourfrac}[2]{\genfrac{}{}{0pt}{}{\scriptstyle #1}{\scriptstyle #2}}
\nc{\oursetminus}{\!\smallsetminus\!}
\nc{\potimes}[1]{^{\otimes #1}}% tensor power
\nc{\sbull}{{\scriptscriptstyle\bullet}}%\mathbf{\cdot}}%{}}
\nc{\smallmatrice}[1]{\left(\begin{smallmatrix} #1 \end{smallmatrix}\right)}
\nc{\then}{\Rightarrow}
\nc{\too}{\mathop{\longrightarrow}\limits}
\nc{\uA}{\underline{A}}
\nc{\unit}{\mathbb{1}}% unit for \otimes

\begin{document}

%------------------------------------------------------------------------------

\title[Separable extensions in tt-geometry]{Separable extensions in tensor-triangular geometry and generalized Quillen stratification}%
\author{Paul Balmer}
\date{2016 February 13}

\address{Paul Balmer, Mathematics Department, UCLA, Los Angeles, CA 90095-1555, USA}
\email{balmer@math.ucla.edu}
\urladdr{http://www.math.ucla.edu/$\sim$balmer}

\begin{abstract}
We exhibit a link between the Going-Up Theorem in commutative algebra and Quillen Stratification in modular representation theory. To this effect, we study the continuous map induced on spectra by a separable extension of tensor-triangulated categories. We determine the image of this map and bound the cardinality of its fibers by the degree of the extension. We then prove a weak form of descent, ``up-to-nilpotence", which allows us to generalize Quillen Stratification to equivariant derived categories.

\bigbreak

Nous exhibons un lien entre le th\'eor\`eme du ``going-up" en alg\`ebre commutative et le th\'eor\`eme de stratification de Quillen en th\'eorie des repr\'esentations modulaires. Dans ce but, nous \'etudions l'application continue induite sur les spectres par une extension s\'eparable de cat\'egories triangul\'ees tensorielles. Nous en d\'eterminons l'image et bornons le cardinal de ses fibres par le degr\'e de l'extension. Nous prouvons alors une forme faible de descente, ``\`a nilpotence pr\`es", qui nous permet de g\'en\'eraliser la stratification de Quillen \`a d'autres cat\'egories d\'eriv\'ees \'equivariantes.
\end{abstract}


\thanks{Research supported by NSF grant~DMS-1303073.}

\maketitle

\vskip-\baselineskip\vskip-\baselineskip
\tableofcontents
\vskip-\baselineskip\vskip-\baselineskip\vskip-\baselineskip

%------------------------------------------------------------------------------

\section*{Introduction}
\medskip
%------------------------------------------------------------------------------

There are rich and growing connections between commutative algebra and modular representation theory, notably via homological methods. Some of these connections can be built using tensor-triangulated categories. Recall that tensor-triangulated categories also appear in several other settings, like motivic theory, equivariant stable homotopy theory, or Kasparov's KK-theory of C*-algebras, for instance. This framework is the backdrop of \emph{tensor-triangular geometry}, or \emph{tt-geometry} for short, see~\cite{BalmerICM}.

In the present work, we use tt-geometry to connect two classical and well-known results, namely the Going-Up Theorem in commutative algebra and Quillen's Stratification Theorem in modular representation theory. Let us first remind the reader\,:

\begin{Thm*}[Going-Up]
Let $R\subset A$ be an integral extension of commutative rings, let $\gq$ be a prime in~$A$ and $\gp'$ a prime in~$R$ containing~$\gq\cap R$; then there exists a prime $\gq'$ in~$A$ containing $\gq$ such that $\gq'\cap R=\gp'$. Further, $\Spec(A)\to \Spec(R)$ is surjective and a weak form of injectivity holds, known as ``Incomparability"\,: if $\gq\subseteq \gq'$ are two primes in~$A$ such that $\gq\cap R=\gq'\cap R$ then $\gq=\gq'$.
\end{Thm*}

\begin{Thm*}[Quillen Stratification~\cite{Quillen71}]
Let $G$ be a finite group and $\kk$ be a field of positive characteristic~$p$ dividing the order of~$G$. Let $\cV_G:=\Proj(\rmH^\sbull(G,\kk))$ be the \emph{projective support variety} of~$G$. Then there is a canonical homeomorphism
\begin{equation}
\label{eq:Quillen}%
\mathop{\colim}_{H\in \Or(G,\,\Abelem)}\!\!\cV_H\ \ \isotoo \ \cV_G
\end{equation}
where $\Or(G,\Abelem)$ is the full subcategory of the orbit category of~$G$ on the elementary abelian $p$-subgroups of~$G$ (see Definition~\ref{def:Or(G)} if necessary).
\end{Thm*}

It is not obvious to the naked eye why the above two results should be related. Let us observe the tip of the iceberg\,: The Going-Up Theorem forces the rings $A$ and~$R$ to have the same Krull dimension; similarly Quillen Stratification forces the Krull dimension of~$\cV_G$ to be the maximum of the dimensions of the~$\cV_H$, \ie the $p$-rank of~$G$ minus one -- an important application of~\cite{Quillen71}.

Here, we prove an analogue of Going-Up in tt-geometry and show that it specializes to Quillen Stratification when applied to modular representation theory. This result illustrates the connections between the two subjects and the depth of tt-geometry. To go beyond unification and connection, we also prove some new results, namely we extend Quillen Stratification to any tensor-triangulated category receiving the derived category of~$G$.

%------------------------------------------------------------------------------

\section{Statement of results}
\medskip
%------------------------------------------------------------------------------

To do tt-geometry, one needs a tt-category~$\cat K$. (Here, ``tt'' is short for ``tensor-triangular" or ``tensor-triangulated", as appropriate.) In commutative algebra, we use $\cat K=\Dperf(R)$, the homotopy category of bounded complexes of finitely generated projective $R$-modules. In modular representation theory, we use $\cat K=\stab(\kk G)$, the stable category of finitely generated $\kk G$-modules modulo the projective ones.

To recover spaces like the affine scheme $\Spec(R)$ in the first example and the support variety $\cV_G$ in the second, we use the \emph{spectrum} $\SpcK$ of a tt-category~$\cat K$. This fundamental tool of tt-geometry was introduced in~\cite{Balmer05a} as the universal topological space in which objects $x$ of~$\cat K$ admit reasonable supports  $\supp(x)\subseteq \SpcK$; see Rem.\,\ref{rem:tt}. The spectrum can also be constructed by means of \emph{prime ideals} $\cat P$ in~$\cat K$. By~\cite{Balmer05a}, it recovers the Zariski spectrum $\Spc(\Dperf(R))\cong \Spec(R)$ and the support variety $\Spc(\stab(\kk G))\cong \cV_G$. This unification is the key to relate affine schemes and support varieties via tt-geometry. See more in Remark~\ref{rem:context}.

To exhibit a first analogy between Going-Up and Quillen Stratification, let us observe that both situations involve not only \emph{one} tt-category but actually \emph{two} (or more). In commutative algebra, it is rather obvious: We have the ring homomorphism $R\to A$, hence an extension-of-scalars $\Dperf(R)\to \Dperf(A)$ which is a tt-functor. In modular representation theory, we also have tt-functors $\Res^G_H:\stab(\kk G)\to \stab(\kk H)$ given by restriction from the group~$G$ to its subgroups~$H\leq G$, for instance the elementary abelian ones. Another key step in our discussion is to understand those restriction functors $\stab(\kk G)\to \stab(\kk H)$ as a form of ``extension-of-scalars", similar to the extension-of-scalars of commutative algebra.

This is made possible thanks to a good notion of ``ring" in a general tt-category~$\cat K$, inspired by commutative algebra but flexible enough to be useful in representation theory as well. These good rings are the ``tt-rings" of~\cite{Balmer14}\,: An associative and unital ring object $A\otimes A\otoo{\mu}A$ in~$\cat K$ is called a \emph{tt-ring} if it is commutative and \emph{separable}, in the classical sense (\ie the multiplication $\mu$ has a section $\sigma:A\to A\otimes A$ which is $A$-linear on both sides). Separability of~$A$ guarantees that the category $\AK$ of good-old $A$-modules in~$\cat K$ remains triangulated, by~\cite{Balmer11} (see Remark~\ref{rem:sep-tri}). Commutativity of~$A$ allows us to equip $\AK$ with a tensor $\otimes_A$. In short, if $A$ is a tt-ring in a tt-category~$\cat K$, then $\AK$ is again a tt-category and extension-of-scalars $F_A:\cat K\too \AK$ is a tt-functor.

By~\cite{Balmer15} the restriction functor $\Res^G_H:\stab(\kk G)\to \stab(\kk H)$ is actually such an extension-of-scalars with respect to the particular tt-ring~$A^G_H:=\kk(G/H)$, with multiplication extending $\kk$-linearly the rule~$\gamma\cdot\gamma=\gamma$ and $\gamma\cdot\gamma'=0$ for all $\gamma\neq\gamma'$ in~$G/H$ (see Constr.\,\ref{cons:AX}). More precisely, there is an equivalence
\begin{equation}
\label{eq:AGH}%
A^G_H\MModcat{{}\cat K(G)}\cong \cat K(H)
\end{equation}
for $\cat K(G)=\stab(\kk G)$ and of course $\cat K(H)=\stab(\kk H)$; under this equivalence, extension-of-scalars $F_{A^G_H}$ becomes isomorphic to restriction~$\Res^G_H$. In particular, the induced map on spectra recovers the usual map on support varieties $\cV_H\to \cV_G$. This recasting of restriction as an extension-of-scalars is already true for derived categories, \ie the above~\eqref{eq:AGH} holds if $\cat K(G)$ means $\Db(\kk G)$, etc. See details in~\cite[Part~I]{Balmer15}. Furthermore we show in~\cite{BalmerDellAmbrogioSanders14pp} that such results hold quite generally for equivariant tt-categories, way beyond representation theory.

In general, for any tt-category~$\cat K$ and any tt-ring~$A$ in~$\cat K$, the tt-functor $F_A:\cat K\to A\MModcat{K}$ induces a continuous map on spectra
\begin{equation}
\label{eq:varphiA}%
\varphi_A:=\Spc(F_A)\ :\quad \Spc(A\MModcat{K})\too \Spc(\cat K)\,.
\end{equation}
It is clearly important to study this map~$\varphi_A$ in general, independently of Going-Up or Quillen Stratification. Here is our main tt-geometric result (see details below):

\begin{Thm}[Descent-up-to-nilpotence, see Thm.\,\ref{thm:nil-desc}]
\label{thm:main-intro}%
Let $\cat K$ be a tt-category and let $A$ be a tt-ring in~$\cat K$. Suppose that $A$ has finite degree and that $A$ detects $\otimes$-nilpotence of morphisms. Then we have a coequalizer of topological spaces
\begin{equation}
\label{eq:Coeq}%
\vcenter{\xymatrix{\Spc(A\potimes{2}\MModcat{K}) \ \ar@<.2em>[r]^-{\varphi_1} \ar@<-.2em>[r]_-{\varphi_2}
& \ \Spc(A\MModcat{K}) \ar[r]^-{\varphi_A}
& \SpcK
}}
\end{equation}
where $\varphi_1$ and $\varphi_2$ are induced by the two obvious homomorphisms $A\xymatrix@C=1em{\ar@<-.2em>[r]\ar@<.2em>[r]&}A\otimes A$.
\end{Thm}

The \emph{degree} of a tt-ring $A$ has also been introduced in~\cite{Balmer14} but can easily be considered as a black-box here. It is the natural measure of the ``size" of the tt-ring~$A$. It will allow us to prove some results by induction on the degree, via Theorem~\ref{thm:split}. Having finite degree is a mild hypothesis which holds for all tt-rings in all standard tt-categories of compact objects in algebraic geometry, homotopy theory or modular representation theory, by~\cite[\S\,4]{Balmer14}. See more in Section~\ref{se:hyp}.

The assumption that $A$ \emph{detects $\otimes$-nilpotence} is the important condition of the theorem. It means that if a morphism $f$ in~$\cat K$ satisfies $A\otimes f\!=\!0$ then $f\potimes{n}\!=\!0$ for some $n\geq1$. This property of the tt-ring~$A$ will be called \emph{nil-faithfulness}. Faithfulness ($A\otimes f=0\,\then\,f=0$) is a special case of interest, as we shall discuss below. Proposition~\ref{prop:solid} provides further equivalent characterizations of nil-faithfulness, like for instance that $A$ generates~$\cat K$ as a thick $\otimes$-ideal, \ie $\supp(A)=\SpcK$.

The strength of~\eqref{eq:Coeq} will be discussed below, when we return to examples. For the moment, simply note that it tells us that $\varphi_A$ is onto and that two points of $\Spc(\AK)$ have the same image only if they are the image of the same point under $\varphi_1$ and~$\varphi_2$. This result is as good as it gets in full tt-generality. In fact, Theorem~\ref{thm:main-intro} is the culmination of a series of results of independent interest, proved in Section~\ref{se:main} and summarized in Theorem~\ref{thm:going-up-intro} below. They illustrate how our study of $\varphi_A$ in tt-geometry mirrors integral extensions in commutative algebra.
\begin{Thm}
\label{thm:going-up-intro}%
Let $A$ be a tt-ring in a tt-category~$\cat K$ and consider $\varphi_A$ as in~\eqref{eq:varphiA}.
\begin{enumerate}
\item
The image of $\varphi_A$ equals the support of $A$, that is, $\Img(\varphi_A)=\supp(A)$ in~$\SpcK$. It will follow that $\varphi_A$ is onto if and only if $A$ is nil-faithful.
\end{enumerate}
Suppose furthermore that $A$ has finite degree~$d$.
\begin{enumerate}
\setcounter{enumi}{1}
\item
Going-Up\,: For every $\cat{Q}\in \Spc(\AK)$ and every $\cat{P}\,'\in \adhpt{\varphi_A(\cat{Q})}$ there exists $\cat{Q}'\in \adhpt{\cat{Q}}$ such that $\varphi_A(\cat{Q}')=\cat{P}\,'$.
\smallbreak
\item
Incomparability\,: If $\cat{Q}'\in \adhpt{\cat{Q}}$ in $\Spc(\AK)$ are such that $\varphi_A(\cat{Q})=\varphi_A(\cat{Q}')$ then $\cat{Q}=\cat{Q}'$.
\smallbreak
\item
The fibers of $\varphi_A$ are finite and discrete, with at most~$d$ points.
\end{enumerate}
The result implies that $\SpcAK$ and $\supp(A)$ have the same Krull dimension. If $A$ is nil-faithful then $\Spc(\AK)$ and $\SpcK$ have the same Krull dimension.
\end{Thm}

\goodbreak

As we shall see, the proofs of these results are quite different from their commutative algebra counterparts. The main obstacle to the tt-generalization of classical techniques is the absence of a tt-quotient, \ie the tt-equivalent of the ring $R/I$ for an ideal $I\subset R$. (Verdier quotients $\cat K/\cat J$ lead to open subschemes, not to closed ones.) In particular, there is no ``residue field" in general tt-geometry, at least for the moment, and this makes the description of the fibers of~$\varphi_A$ particularly difficult.

\smallbreak

To conclude our discussion, we still want to explain why Theorem~\ref{thm:main-intro} is called ``descent-up-to-nilpotence" and understand nil-faithfulness in examples.

\smallbreak

Conceptually, the category $\AK$ should be approached as a \emph{simplification} of the original category~$\cat K$. In some sense, it represents a ``finite \'etale extension of~$\cat K$". This idea goes hand-in-hand with Grothendieck's theory of \emph{descent}, \ie the problem of reconstructing the original category~$\cat K$ from $\AK$ and data over $A\potimes{2}$ and~$A\potimes{3}$. By~\cite{Balmer12}, if $A$ was not just nil-faithful but really faithful ($A\otimes f\!=\!0\,\then\, f\!=\!0$) then $A$ would satisfy descent in~$\cat K$, which roughly says that
$$\xymatrix@C=2em{\cat K \ar[r]
& \ A\MModcat{K} \ar@<.2em>[r] \ar@<-.2em>[r]
& \ A\potimes{2}\MModcat{K}
}$$
is an ``equalizer of categories" (plus cocycle conditions over~$A\potimes{3}$). In that case, it is easier to prove, and even easier to believe, that the contravariant $\Spc(-)$ turns the latter ``equalizer" into the coequalizer of spaces~\eqref{eq:Coeq}. The remarkable fact about Theorem~\ref{thm:main-intro} is that~\eqref{eq:Coeq} is a coequalizer even when $A$ is not faithful but only \emph{nil}-faithful, despite failure of descent in this situation! This explains the idea of ``descent-up-to-nilpotence".

To fully appreciate the difference between faithful and nil-faithful, \ie between descent and descent-up-to-nilpotence, we turn to our example of modular representation theory, where we will also better understand the coequalizer~\eqref{eq:Coeq}.

Let us test the distinction between faithful and nil-faithful on the tt-rings $A^G_H=\kk(G/H)$ that we discussed above, see~\eqref{eq:AGH}. We assume for simplicity that $G$ is a $p$-group, where $p=\chara(\kk)>0$. Then consider in $\cat K=\stab(\kk G)$, the tt-ring
$$
\qquad A:=\kk(G/H_1)\times\cdots\times\kk(G/H_n)\qquad\textrm{for subgroups }H_1,\dots,H_n\leq G\,.
$$
This $A$ is faithful only if one of the $H_i$ equals~$G$, which is sadly restrictive since $\AK$ contains a copy of~$\cat K$ in that case. On the other hand, this same $A$ is nil-faithful much more often, most famously if $H_1,\ldots,H_n$ are the elementary abelian $p$-subgroups of~$G$. In fact, nil-faithfulness of this particular tt-ring~$A$ is Serre's old theorem on the vanishing of the Bocksteins~\cite{Serre65}; see Thm.\,\ref{thm:Serre}. In other words, descent happens only in the trivial case where $\cat K$ is a summand of $\AK$ but descent-up-to-nilpotence happens much more often. In fact, in that case, our descent-up-to-nilpotence Theorem~\ref{thm:main-intro} recovers Quillen Stratification.

Although our goal is \emph{not} to give a new proof of Quillen's result, we hope that recovering such a major result as an example could help the reader grasp the power of Theorem~\ref{thm:main-intro}. To provide new applications, we extend Quillen Stratification to every tt-category receiving the derived category. This is Theorem~\ref{thm:Quillen} below\,:
\begin{Thm}
Let $p$ be a prime and $G$ a finite group with $p$ dividing~$|G|$. Let $\cat K(G)$ be a tt-category which admits a tt-functor $\Phi:\Db(\bbF_{\!p}G)\too \cat K(G)$, where $\bbF_{\!p}$ is the finite field with~$p$ elements. For every subgroup~$H\leq G$, consider the tt-ring $A_H=\Phi(\bbF_{\!p}(G/H))$, image in $\cat K(G)$ of the complex $\bbF_{\!p}(G/H)\in\Db(\bbF_{\!p}G)$ concentrated in degree zero. Define $\cat K(H)=A_H\MModcat{K(G)}$ the corresponding category of modules. Then the tt-functors $F_{A_H}:\cat K(G)\too \cat K(H)$ induce a homeomorphism
\begin{equation}
\label{eq:Colim}%
\mathop{\colim}_{H\in \Or(G,\,\Abelem)}\Spc(\cat K(H))\ \isotoo \ \Spc(\cat K(G))\,.
\end{equation}
\end{Thm}

The reason for the notation $\cat K(H)$ is that, in most examples, the category $\cat K(G)$ is defined for all finite groups~$G$ together with an equivalence $A_H\MModcat{K(G)}\cong \cat K(H)$, as in~\eqref{eq:AGH}. As we show in~\cite{BalmerDellAmbrogioSanders14pp}, such ``monadicity" results happen in most equivariant settings, way beyond representation theory. They give~\eqref{eq:Colim} its full flavor since in that case the left-hand category $\cat K(H)$ is not merely a category of modules in~$\cat K(G)$ but the true object of study over the subgroup~$H$. To give an example for which $\Spc(\cat K(G))$ is not known, one could take $\cat K(G)$ to be the bounded derived category of equivariant vector bundles over a variety with $G$-action. In that case, Quillen Stratification is new. See Example~\eqref{exa:equiv}.

\smallbreak

The organization of the paper should now be clear from the table of content.

%------------------------------------------------------------------------------
\goodbreak
\section{Recalling tt-rings and their degree}
\label{se:hyp}%
\medbreak
%------------------------------------------------------------------------------

%
\begin{Conv}
\label{conv:tc}%
We use the word \emph{triangulated} in the sense of~\cite[\S\,5]{Balmer11}, or K\"unzer~\cite{Kuenzer07} and Maltsiniotis~\cite{Maltsiniotis06}. This is a mild strengthening of Verdier's traditional definition in which we require octahedra to satisfy the analogue of the morphism axiom. One can also require higher triangles but we do not need them here. All examples (stable homotopy categories) are triangulated in that sense.
\end{Conv}

\begin{Def}
A \emph{tt-category} $\cat K$ is an essentially small, idempotent-complete tensor-triangulated category. Hence $\cat K$ is triangulated, all idempotent endomorphisms in~$\cat K$ split, and $\cat K$ is equipped with a symmetric monoidal structure $\otimes:\cat K\times \cat K\too\cat K$ such that $x\otimes-:\cat K\to \cat K$ is exact for all objects~$x\in\cat K$.
We denote by $\unit\in\cat K$ its $\otimes$-unit. Throughout this paper, $\cat K$, $\cat L$, \ldots\ denote tt-categories. A \emph{tt-functor} $F:\cat K\to \cat L$ is a functor which is both exact and monoidal, in such a way that the two natural isomorphisms from $F(x\otimes \Sigma y)$ to $\Sigma F(x\otimes y)$ agree (using in different order the compatibility of~$F$ with~$\otimes$, of~$F$ with $\Sigma$, and of~$\otimes$ with~$\Sigma$) and similarly on the left.
\end{Def}

\begin{Exa}
\label{exa:Dperf}%
Let $X$ be a quasi-compact and quasi-separated scheme (\eg\ a noetherian scheme or an affine scheme) then the derived category $\cat K=\Dperf(X)$ of perfect complexes~\cite{SGA6} is a tt-category. For $X=\Spec(R)$, this is the tt-category $\Dperf(R)$ of the introduction.
\end{Exa}

\begin{Exa}
\label{exa:kGstab}%
Let $G$ be a finite group and $\kk$ a field of positive characteristic~$p$ dividing~$|G|$. Then the derived category $\DbG:=\Db(\kk G\mmod)$ and the stable category $\stab(\kk G)$ of finitely generated $\kk G$-modules are tt-categories and the well-known (Buchweitz-Rickard) quotient $\DbG\onto\stab(\kk G)$ is a tt-functor. See~\cite{Rickard89}.
\end{Exa}

\begin{Not}
\label{not:leq}%
We write $x\leq y$ when $x$ is a direct summand of~$y$. A triangulated subcategory $\cat J\subseteq\cat K$ is called a \emph{thick $\otimes$-ideal} if $x\leq y\in\cat J\then x\in\cat J$ and if $\cat K\otimes\cat J\subseteq\cat J$.
We denote by $\ideal{\cat E}$ the thick $\otimes$-ideal generated by a collection of objects~$\cat E\subseteq\cat K$.
\end{Not}

\begin{Rem}
\label{rem:tt}%
We do not need much tt-geometry here but we simply recall that the \emph{spectrum} $\SpcK$ of a tt-category~$\cat K$ consists of all \emph{prime} thick $\otimes$-ideals $\cat P\subsetneq\cat K$, \ie such that $x\otimes y\in\cat P$ implies $x\in\cat P$ or~$y\in\cat P$. The topology of~$\SpcK$ has an open basis composed of all subsets $\cat U(x):=\SET{\cat P}{x\in\cat P}$, for $x\in\cat K$. The closed complement $\supp(x):=\SET{\cat P}{x\notin\cat P}$ is the \emph{support} of~$x$ in~$\SpcK$.
\end{Rem}

\begin{Rem}
\label{rem:context}%
Determining tt-spectra is an ongoing enterprise in several fields of mathematics, with direct link to the classifications of thick $\otimes$-ideals. This program originated in stable homotopy theory~\cite{HopkinsSmith98}, and extended to algebraic geometry~\cite{Hopkins87,Neeman92a,Thomason97} and modular representation theory~\cite{BensonCarlsonRickard97,FriedlanderPevtsova07} prior to~\cite{Balmer05a}. More recent progress has been made in noncommutative topology~\cite{DellAmbrogio10}, for Artin-Tate motives~\cite{Peter13}, for noncommutative motives~\cite{DellAmbrogioTabuada12}, for group rings over commutative algebras~\cite{StevensonApp13}, for category algebras~\cite{Xu14} and for Lie superalgebras~\cite{BoeKujawaNakano14pp}. Still, computing $\SpcK$ remains an open challenge for many important tt-categories~$\cat K$, like equivariant derived categories or equivariant stable homotopy categories~\cite{GreenleesMay95}.
\end{Rem}

One basic tt-result that we shall need a few times is the following\,:

\begin{Lem}[Existence Lemma {\cite[Lem.\,2.2]{Balmer05a}}]
\label{lem:exist}%
Let $\cat K$ be a tt-category, $\cat J\subseteq\cat K$ a thick $\otimes$-ideal and $S\subset \cat K$ a $\otimes$-multiplicative collection of objects ($\unit\in S\supseteq S\otimes S$) such that $S\cap \cat J=\varnothing$. Then there exists $\cat P\in\SpcK$ such that $\cat J\subseteq \cat P$ and $\cat P\cap S=\varnothing$.
\end{Lem}

\begin{Not}
Given $n$ objects $x_1,\ldots,x_n$ in $\cat K$ and a permutation $\pi\in S_n$, we also denote by $\pi:x_1\otimes\cdots\otimes x_n\isoto x_{\pi(1)}\otimes\cdots\otimes x_{\pi(n)}$ the induced isomorphism.
\end{Not}

\begin{Def}
\label{def:ring}%
A \emph{ring object} $A$ in~$\cat K$ is a monoid $(A\,,\,\mu:A\otimes A\to A\,,\,\eta:\unit\to A)$, \ie admits associative multiplication $\mu:A\otimes A\to A$ and two-sided unit~$\eta$. It is \emph{commutative} if $\mu\circ(12)=\mu$. A \emph{morphism of (commutative) ring objects}, or simply \emph{a homomorphism}, is a morphism $f:A\to B$ in~$\cat K$, compatible with multiplications and units. We also say that $B$ is an \emph{$A$-algebra} or that $B$ is a \emph{ring over}~$A$.
\end{Def}

\begin{Not}
\label{not:prod}%
For two ring objects $A$ and $B$, the ring object $A\times B$ is $A\oplus B$ with component-wise structure. On the other hand, the ring object $A\otimes B$ has multiplication $(\mu_1\otimes\mu_2)\circ (23):(A\otimes B)\potimes{2}\too A\otimes B$ and obvious unit.
\end{Not}

\begin{Def}
\label{def:sep}%
A ring object $A$ is called \emph{separable} if there exists a morphism $\sigma:A\to A\otimes A$ such that $\mu\sigma=\id_A$ and $(1\otimes \mu)\circ(\sigma\otimes 1)=\sigma\mu=(\mu\otimes 1)\circ(1\otimes \sigma):A\potimes{2}\to A\potimes{2}$. This simply means that $A$ is projective as $A,A$-bimodule.
\end{Def}

\begin{Rem}
\label{rem:tens-sep}%
If $A$ and $B$ are separable then so are $A\times B$ and $A\otimes B$.
\end{Rem}

\begin{Rem}
\label{rem:sep-tri}%
A \emph{(left) $A$-module} is a pair $(x,\rho)$ where $x\in\cat K$ and where the so-called action $\rho:A\otimes x\to x$ satisfies the usual associativity and unit conditions. We denote by $\AK$ the category of $A$-modules with $A$-linear morphisms and by $F_A:\cat K\to \AK$ the functor $F_A(y)=(A\otimes y,\mu\otimes 1)$, etc. For details see~\cite{EilenbergMoore65} or~\cite{Balmer11}.

By~\cite{Balmer11}, when $A$ is separable, the category of $A$-modules in~$\cat K$ remains triangulated in such a way that both the extension-of-scalars functor $F_A:\cat K\too A\MModcat{K}$ and the forgetful functor $U_A:A\MModcat{K}\too \cat K$ are exact. We really mean ``modules in the homotopy category" here, not ``homotopy category of modules", thanks to separability! The key fact is that the co-unit $\epsilon:F_AU_A\to \Id_{\AK}$ is split surjective. In particular, every $A$-module $x$ is a direct summand of a free one\,: $x\leq F_AU_A(x)$.
\end{Rem}

The following simple definition~\cite{Balmer14} plays a central role in our theory.

\begin{Def}
\label{def:tt}%
A ring object~$A$ is a \emph{tt-ring} if $A$ is commutative and separable.
\end{Def}

The terminology is chosen so that if $A$ is a tt-ring in a tt-category~$\cat K$ then $\AK$ remains a tt-category. See details in~\cite{Balmer14}, notably about the tensor structure $\otimes_A$ \emph{over~$A$} on the category~$\AK$. The projection formula~\cite[Prop.\,1.2]{Balmer14} says that $U_A(x\otimes_A F_A(y))\cong U_A(x)\otimes y$ for all $x\in A\MModcat{K}$ and $y\in \cat K$.

\begin{Rem}
In commutative algebra, an $R$-algebra $A$ will not be a tt-ring in $\Dperf(R)$ in general, unless $A$ is finite \'etale over~$R$. In that case, the tt-category $\Dperf(A)$ coincides with $A\MMod_{\Dperf(R)}$ by~\cite[Cor.\,6.6]{Balmer11}. Thus the reader might prefer to think of our extensions-of-scalars $F_A:\cat K\to \AK$ along tt-rings as finite \'etale morphisms.
\end{Rem}

\begin{Rem}
\label{rem:smash}%
Suppose that $\cat K=\cat T^c$ consists of the compact objects of a compactly generated (big) tensor-triangulated category~$\cat T$. Then any smashing localization of~$\cat T$ is in fact an extension-of-scalars with respect to a tt-ring~$A$ in~$\cat T$, which is moreover an idempotent. Conversely, any ring object such that $\mu:A\otimes A\to A$ is an isomorphism yields a smashing localization. This provides another class of examples of separable extensions. In this paper, we restrict attention to compact tt-rings $A\in \cat K$, in order to speak of $\supp(A)$, thus excluding smashing localizations. Yet, some of our results might extend to ``big" tt-rings $A\in \cat T$ with a suitable notion of ``big support".
\end{Rem}

\begin{Rem}
\label{rem:deg}%
We refer to~\cite{Balmer14} for the notion of the \emph{degree} of a tt-ring~$A$. We use it as a black box here, emphasizing that we do not know any tt-ring of infinite degree. (Even the big smashing rings of Remark~\ref{rem:smash} actually have degree~1.) In particular the ``equivariant" tt-rings $\kk(G/H)$ of~\cite{Balmer15} are of finite degree~$[G\!:\!H]$ in $\Db(\kk G)$. Furthermore, the degree cannot increase along tt-functors. The key result allowing induction on the degree is the following~\cite[Thm.\,3.6]{Balmer14}\,:
\end{Rem}

\begin{Thm}
\label{thm:split}%
Let $A$ be a tt-ring of finite degree~$d$ in a tt-category~$\cat K$. Then there exists a tt-ring~$C$ and a homomorphism $g:A\otimes A\to C$ such that $\smallmatrice{\mu\\g}:A\otimes A\isoto A\times C$ is an isomorphism. Moreover, $C$ is of degree~$d-1$ as a tt-ring in~$\AK$, when viewed as an $A$-algebra through $g\circ(1\otimes\eta):A\to C$.
\end{Thm}

%------------------------------------------------------------------------------
\goodbreak
\section{Descent up-to-nilpotence}
\label{se:main}%
\medbreak
%------------------------------------------------------------------------------

%
\begin{Not}
\label{not:varphiA}%
Given a tt-ring $A$ in a tt-category~$\cat K$, we have a continuous (and spectral) map $\varphi_A:=\Spc(F_A)\,:\ \SpcAK\too \SpcK$, as in~\eqref{eq:varphiA}. It is defined by $\varphi_A(\cat P)=F_A\inv(\cat P)$ for every prime $\cat P\subset\AK$. See~\cite[\S\,3]{Balmer05a}. See Remark~\ref{rem:sep-tri} for the functors $F_A:\cat K\adjto \AK:U_A$.
\end{Not}

\begin{Lem}
\label{lem:UJ}%
For every thick $\otimes$-ideal $\cat J_0\subseteq \cat K$, the thick subcategory $U_A\inv(\cat J_0)$ of $\AK$ is $\otimes$-ideal and equals the thick $\otimes$-ideal $\ideal{F_A(\cat J_0)}$ generated by~$F_A(\cat J_0)$.
\end{Lem}

\begin{proof}
Since $A$ is separable, we have $z\leq F_AU_A(z)$ for every $z\in \AK$ (Remark~\ref{rem:sep-tri}). Now, if $y\in U_A\inv(\cat J_0)$ and $x\in\cat K$ then $U_A(y)\in \cat J_0$ and since $\cat J_0$ is $\otimes$-ideal, $\cat J_0\ni x\otimes U_A(y)\simeq U_A(F_A(x)\otimes_A\,y)$. So $F_A(x)\otimes_A\,y\in U_A\inv(\cat J_0)$. Hence for every $z\in\AK$, we have $U_A\inv(\cat J_0)\ni F_A(U_A(z))\otimes_A\,y\geq z\otimes_A\,y$ hence $z\otimes_A\,y\in U_A\inv(\cat J_0)$ as wanted. So, $U_A\inv(\cat J_0)$ is $\otimes$-ideal. Let us show that $U_A\inv(\cat J_0)$ is the smallest thick $\otimes$-ideal containing $F_A(\cat J_0)$. On the one hand, $U_AF_A(x)\simeq A\otimes x$ belongs to $\cat J_0$ for every $x\in\cat J_0$. So, $F_A(\cat J_0)\subseteq U_A\inv(\cat J_0)$ and therefore $\ideal{F_A(\cat J_0)}\subseteq U_A\inv(\cat J_0)$ by the above discussion. Conversely, if $z\in U_A\inv(\cat J_0)$ then $F_A(U_A(z))\in F_A(\cat J_0)$ belongs to $\ideal{F_A(\cat J_0)}$, hence so does every direct summand of $F_A(U_A(z))$, like $z$ itself.
\end{proof}

\begin{Lem}
\label{lem:U}%
Let $\cat P\in \SpcAK$ and let $s\in\AK$ be an object such that $U_A(s)\in \varphi_A(\cat P)$. Then $s\in \cat P$.
\end{Lem}

\begin{proof}
$U_A(s)\in \varphi_A(\cat P)=F_A\inv(\cat P)$ reads $\cat P\ni F_A(U_A(s))\geq s$ and $\cat P$ is thick.
\end{proof}

\begin{Thm}
\label{thm:varphi}%
Let $A$ be a tt-ring in a tt-category~$\cat K$. Then\,:
\begin{enumerate}[\rm(a)]
\item
\label{it:x}%
For every $x\in\cat K$, we have $\varphi_A\inv(\supp(x))=\supp(F_A(x))$ in $\SpcAK$.
\smallbreak
\item
\label{it:S}%
Let $S\subset \AK$ be a $\otimes$-multiplicative collection of objects and consider the closed subset $\cat{Z}(S)=\cap_{s\in S}\supp(s)$ in $\SpcAK$. Then its image is closed too; more precisely $\varphi_A(\cat{Z}(S))=\cat{Z}(U_A(S))$ where $U_A(S)=\SET{U_A(s)}{s\in S}$ (not necessarily $\otimes$-multiplicative). In particular, $\varphi_A$ is a closed map.
\smallbreak
\item
\label{it:U}%
For all $y\in \AK$, we have $\Displ\varphi_A(\supp(y))=\mathop{\cap}_{n\geq1}\supp(U_A(y\potimes{n}))$ in~$\SpcK$. (\,\footnote{\,When $\cat K$ is rigid, or $A$ has finite degree, one can prove $\varphi_A(\supp(y))=\supp(U_A(y))$.})
\smallbreak
\item
\label{it:image}%
The image of $\varphi_A:\SpcAK\to \SpcK$ is exactly the support of~$A$.
\end{enumerate}
\end{Thm}

\begin{proof}
Part~(\ref{it:x}) is a general property of $\Spc(F)$ for any tt-functor~$F$, see~\cite[Prop.\,3.6]{Balmer05a}. For~\eqref{it:S}, note that $\cat{Z}(S)=\SET{\cat P}{S\cap \cat P=\varnothing}$. It follows from Lemma~\ref{lem:U} that $\varphi_A(\cat{Z}(S))\subseteq \cat{Z}(U_A(S))$. Conversely, suppose that $\cat P_0\in \SpcK$ is such that $U_A(s)\notin\cat P_0$ for all $s\in S$ $(\ast)$ and consider the thick $\otimes$-ideal $\cat J:=U_A\inv(\cat P_0)$ in~$\AK$ (Lemma~\ref{lem:UJ}). Consider also the $\otimes$-multiplicative subset $T:=S\otimes_A \SET{F_A(x)}{x\in\cat K\oursetminus\cat P_0}$ in~$\AK$. We claim that $T\cap \cat J=\varnothing$. Indeed, suppose {\sl ab absurdo}, that there exists $s\in S$ and $x\in \cat K\oursetminus \cat P_0$ such that $s\otimes_A F_A(x)\in \cat J=U_A\inv(\cat P_0)$. Then $\cat P_0\ni U_A(s\otimes_A F_A(x))\simeq U_A(s)\otimes x$. Since $x\notin\cat P_0$ and $\cat P_0$ is prime, this implies that $U_A(s)\in\cat P_0$ which contradicts~$(\ast)$. By the Existence Lemma~\ref{lem:exist}, there exists a prime $\cat P\subset \AK$ such that $\cat P\supseteq\cat J$ and $\cat P\cap T=\varnothing$. These properties imply
\centerline{(1) $\cat P\supseteq U_A\inv(\cat P_0)$,
\qquad (2) $\cat P\cap S=\varnothing$
\qquad (3) $\cat P\cap F_A(\cat K\oursetminus\cat P_0)=\varnothing$.}
Since $U_A\inv(\cat P_0)=\ideal{F_A(\cat P_0)}$ by Lemma~\ref{lem:UJ}, the first relation above implies $\cat P_0\subseteq F_A\inv(\cat P)$. Combining with~(3), which implies $F_A\inv(\cat P)\subseteq \cat P_0$, we get $\cat P_0=F_A\inv(\cat P)=\varphi_A(\cat P)$. Finally, (2) reads $\cat P\in \cat{Z}(S)$ hence $\cat P_0=\varphi_A(\cat P)\in\varphi_A(\cat{Z}(S))$ as wanted.

For~\eqref{it:U}, let $S=\{y\potimes{n}, n\geq1\}$ so that $\cat{Z}(S)=\supp(y)$. By~\eqref{it:S}, we deduce $\varphi_A(\supp(y))=\cap_{n\geq1}\supp(U_A(y\potimes{n}))$.

Part~\eqref{it:image} follows from~\eqref{it:U} applied to $y=\unit_A$ since $U_A(\unit_A)=A$.
\end{proof}

\begin{Thm}[Going-Up]
\label{thm:going-up}%
Let $A$ be a tt-ring in~$\cat K$. Let $\cat{Q}\in\SpcAK$ be a point and $\cat{P}:=\varphi_A(\cat{Q})$ its image in~$\SpcK$. Let $\cat{P}'\in \adhpt{\cat{P}}$ be a point in the closure of~$\cat{P}$. Then there exist $\cat{Q}'\in\adhpt{\cat{Q}}$ in the closure of~$\cat{Q}$ such that $\varphi_A(\cat{Q}')=\cat{P}'$. In cash\,: if $\cat{P}'\subseteq \cat{P}=F_A\inv(\cat{Q})$ then there exists $\cat{Q}'\subseteq\cat{Q}$ such that $F_A\inv(\cat{Q}')=\cat{P}'$.
\end{Thm}

\begin{proof}
This follows by Theorem~\ref{thm:varphi}\,\eqref{it:S} with $S=(\AK)\oursetminus\cat{Q}$. Indeed, $\adhpt{\cat{Q}}=\cat{Z}(S)$ for that particular~$S$ (see \cite[Prop.\,2.9]{Balmer05a}) and $\cat{P}'\in \cat{Z}(U_A(S))$ since for every $s\notin\cat{Q}$, Lemma~\ref{lem:U} implies that $U_A(s)\notin \varphi_A(\cat{Q})=\cat{P}\supseteq\cat{P}'$ so $U_A(s)\notin\cat{P}'$.
\end{proof}

\begin{Rem}
\label{rem:barF}%
Let $F:\cat K\to \cat L$ be a tt-functor, $A$ a tt-ring in~$\cat K$ and $B:=F(A)$ the image tt-ring in~$\cat L$. Then there exists a tt-functor $\bar F:A\MModcat{K}\too B\MModcat{L}$ such that $\bar FF_A=F_BF$ and $U_B\bar F=F\,U_A$ (see~\cite[Rem.\,1.6]{Balmer14})\,:
\begin{equation}
\label{eq:barF}%
\vcenter{\xymatrix@R=2em{\cat K \ar[r]^-{F} \ar@<-.2em>[d]_-{F_A}
& \cat {L} \ar@<-.2em>[d]_-{F_B}
\\
A\MModcat{K} \ar@<-.2em>[u]_-{U_A} \ar[r]^-{\bar F}
& B\MModcat{L} \ar@<-.2em>[u]_-{U_B}
}}
\end{equation}

In two places below, we shall apply this construction to $F=F_A$ itself. To avoid confusion, let us say that $A_1,A_2$ are two tt-rings in~$\cat K$ (later to be $A_1=A_2=A$). In the above notation, we set $A=A_2$, $\cat L=A_1\MModcat{K}$ and $F=F_{A_1}$. So the ring $B=F_{A_1}(A_2)$ in~$\cat L$ has underlying ring $A_1\otimes A_2$ in~$\cat K$. By~\cite[Rems.\,1.4 and~1.5]{Balmer14}, we have an identification $B\MModcat{L}\cong (A_1\otimes A_2)\MModcat{K}$ under which the two functors $F_B$ and $\bar F$ become the two functors $A_1\MModcat{K}\to (A_1\otimes A_2)\MModcat{K}$ and $A_2\MModcat{K}\to (A_1\otimes A_2)\MModcat{K}$ associated to the two homomorphisms $1\otimes \eta_2: A_1\to A_1\otimes A_2$ and $\eta_1\otimes 1:A_2\to A_1\otimes A_2$.
\end{Rem}

\begin{Lem}
\label{lem:pb}%
Let $F:\cat K\to \cat L$ be a tt-functor, $A$ a tt-ring in~$\cat K$ and $B=F(A)$ as in Remark~\ref{rem:barF}. Then diagram~\eqref{eq:barF} yields a commutative square of spectra\,:
\begin{equation}
\label{eq:pb}%
\vcenter{\xymatrix@C=3em{
& \SpcK
&& \Spc(\cat L) \ar[ll]_-{\varphi=\Spc(F)}
& \cat Q \ar@{}[l]|-{\ni}
\\
\cat P \ar@{}[r]|-{\in}
& \Spc(A\MModcat{K}) \ar[u]^-{\varphi_A=\Spc(F_A)}
&& \Spc(B\MModcat{L})\,. \ar[ll]_-{\bar \varphi=\Spc(\bar F)} \ar[u]_-{\varphi_B=\Spc(F_B)}
}}
\end{equation}
Let $\cat P\in\Spc(\AK)$ and $\cat Q\in\Spc(\cat{L})$ such that $\varphi_A(\cat P)=\varphi(\cat Q)$. Then there exists $\cat Q'\in \Spc(B\MModcat{L})$ such that $\bar\varphi(\cat Q')\in\adhpt{\cat P}$ (\ie $\bar\varphi(\cat Q')\subseteq\cat P$) and $\varphi_B(\cat Q')=\cat Q$.
\end{Lem}

\begin{proof}
Consider the $\otimes$-multiplicative $S\subset B\MModcat{L}$ given by $S=\SET{\bar F(s)}{s\in\AK,\ s\notin\cat P}$. We claim that $\cat Q\cap U_B(S)=\varnothing$. Indeed, if $\cat Q\ni U_B(\bar F(s))\simeq F(U_A(s))$ for some $s\notin\cat P$ then $U_A(s)\in F\inv(\cat Q)=\varphi(\cat Q)=\varphi_A(\cat P)$ hence $s\in\cat P$ by Lemma~\ref{lem:U}, which contradicts the choice of~$s$. In other words, $\cat Q\in \cat{Z}(U_B(S))$ and by Theorem~\ref{thm:varphi}\,\eqref{it:S} applied to the tt-ring~$B$ in the tt-category~$\cat L$, there exists $\cat Q'\in \cat{Z}(S)$ such that $\varphi_B(\cat Q')=\cat Q$. The property $\cat Q'\in \cat{Z}(S)$ means that $\bar F((\AK)-\cat P)\cap \cat Q'=\varnothing$ which means that $\bar\varphi(\cat Q')=\bar F\inv(\cat Q')\subseteq \cat P$, as claimed.
\end{proof}

\begin{Thm}
\label{thm:main}%
Let $A\in\cat K$ be a tt-ring of \emph{finite degree}~$d\geq1$ (see Remark~\ref{rem:deg}).
\begin{enumerate}[\rm(a)]
\item
\label{it:incomp}%
Incomparability\,: Let $\cat P,\cat P'\in \SpcAK$ be two primes such that $\cat P\subseteq\cat P'$ (\ie $\cat P\in\adhpt{\cat P'}$) and $\varphi_A(\cat P)=\varphi_A(\cat P')$. Then $\cat P=\cat P'$.
\smallbreak
\item
\label{it:Krull}%
The space $\Spc(\AK)$ has the same Krull dimension as $\supp(A)$ -- the latter as a subspace of~$\SpcK$.
\smallbreak
\item
\label{it:finite}%
The fibers of the map $\varphi_A:\Spc(\AK)\to\SpcK$ are discrete (have Krull dimension zero) and contain at most $d$ points.
\end{enumerate}
\end{Thm}

\begin{proof}
Note right away that~\eqref{it:Krull} follows from part~\eqref{it:incomp}, Theorem~\ref{thm:varphi}\,\eqref{it:image} and Going-Up Theorem~\ref{thm:going-up}. Similarly, \eqref{it:incomp} implies the discretion of the fibers in~\eqref{it:finite}. So let us prove Incomparability and the bound on the cardinality of the fibers by induction on $d=\deg(A)$. For $d=0$, we have $A=\unit$ and there is nothing to prove. This being essentially a convention, let us discuss the next case. For $d=1$, we use~\cite[Prop.\,4.1]{Balmer14} which tells us that the multiplication $\mu:A\otimes A\isoto A$ is an isomorphism and $F_A$ is simply a (Bousfield, smashing) localization, in such a way that the map $\varphi_A:\Spc(\AK)\to \SpcK$ is the inclusion of an open and closed component of~$\SpcK$ corresponding to~$\supp(A)$; again in this case the results are straightforward. Suppose $\deg(A)>1$ and the result true for tt-rings of degree~$d-1$.

Let us do a preparation useful for~\eqref{it:incomp} and~\eqref{it:finite}. By Theorem~\ref{thm:split}, there exists a ring isomorphism $h=\smallmatrice{\mu\\g}:A\otimes A\isoto A\times C$ such that $C$ has degree $d-1$, when viewed as tt-ring in~$\AK$. To be careful, let us write $\bar C\in\AK$ for $C$ viewed as an $A$-module via $g\,(1\otimes \eta)$, so that $C=U_A(\bar C)$. See~\cite[Rem.\,1.5]{Balmer14}, which also recalls the canonical equivalence $\bar C\MModcat{L}\cong C\MModcat{K}$ where $\cat L:=\AK$ in such a way that $F_{\bar C}\circ F_A\cong F_C$. Consider now the left-hand diagram of tt-rings below (defining the two homomorphisms $g_i:A\to C$ for $i=1,2$)\,:
$$
\xymatrix{
\unit \ar[r]^{\eta} \ar[d]_{\eta}
& A \ar[d]^-{\eta\otimes 1} \ar@/^1em/[rdd]^-{\smallmatrice{1\\g_2}}
\\
A \ar[r]^-{1\otimes\eta} \ar@/_1em/[rrd]_-{\smallmatrice{1\\g_1}}
& A\otimes A \ar[rd]^-{h}_-{\simeq}
\\
&& A\times C
}\qquad
\xymatrix{
\cat K \ar[r]^{F_A} \ar[d]_{F_A}
& \AK \ar[d]^-{F_2} \ar@/^2em/[rdd]^-{\smallmatrice{\Id\\G_2}}
\\
\AK \ar[r]^-{F_1} \ar@/_1em/[rd]_-{\smallmatrice{\Id\\G_1}}
& (A\otimes A)\MModcat{K} \ar[rd]^-{H}_-{\simeq}
\\
&& \kern-6em\AK\times C\MModcat{K}
}
$$
These homomorphisms induce the above right-hand diagram of tt-functors, which in turn induces the following commutative diagram of topological spaces\,:
\begin{equation}
\label{eq:main}%
\vcenter{\xymatrix{
& \SpcK
& \Spc(\AK) \ar[l]_-{\varphi_A}
\\
& \Spc(A\MModcat{K}) \ar[u]^-{\varphi_A}
& \Spc(A\potimes{2}\MModcat{K}) \ar[l]_-{\varphi_1} \ar[u]^-{\varphi_2} \ar@{<-}[rd]_-{\simeq}
\\
&& \ar@/^1em/[lu]^(.6){\Displ(\id\ \psi_1)\ }
& \kern-6em\Spc(\AK)\sqcup \Spc(C\MModcat{K})  \ar@/_2em/[luu]_-{\Displ(\id\ \psi_2)}
}}\kern-1em
\end{equation}
where $\varphi_i=\Spc(F_i)$ and $\psi_i=\Spc(G_i)$, for $i=1,2$. Note that the upper-left square in~\eqref{eq:main} is nothing but~\eqref{eq:pb} applied to $F=F_A$ itself. (See the explanations in the second half of Remark~\ref{rem:barF}.) By the construction of~$C$ and the discussion above, the functor $G_1:\AK\to C\MModcat{K}\cong\bar C\MModcat{L}$ is just an extension-of-scalar $F_{\bar C}$ with respect to the tt-ring $\bar C$ in~$\AK$, which has degree~$d-1$. We can therefore apply the induction hypothesis to the map~$\psi_1$, \ie we can assume that~\eqref{it:incomp} and~\eqref{it:finite} hold for~$\psi_1$.

Before jumping into the proof of~\eqref{it:incomp}, note that since $\varphi_A=F_A\inv(-)$ commutes with arbitrary intersection, an easy application of Zorn's Lemma shows that any $\varphi_A$-fiber admits a minimal prime inside any given~$\cat P$. So, we can assume $\cat P$ minimal in its $\varphi_A$-fiber\,: If $\tilde{\cat P}\subseteq\cat P$ and $\varphi_A(\tilde{\cat P})=\varphi_A(\cat P)$ then $\tilde{\cat P}=\cat P$.
By Lemma~\ref{lem:pb}, there exists $\cat Q'\in \Spc(A\potimes{2}\MModcat{K})$ such that $\varphi_1(\cat Q')\in\adhpt{\cat P}$ and $\varphi_2(\cat Q')=\cat P'$. In the decomposition of $\Spc(A\potimes{2})$ as $\Spc(A)\sqcup\Spc(C)$, we have two possibilities. Either $\cat Q'\in\Spc(A)$, in which case (since $\varphi_1=\varphi_2=\id$ on the $\Spc(A)$ part) we get $\cat P'=\varphi_2(\cat Q')=\cat Q'=\varphi_1(\cat Q')\in\adhpt{\cat P}$, meaning $\cat P'\subseteq\cat P$, and we have $\cat P=\cat P'$ as wanted. Or $\cat Q'\in\Spc(C)$. In that case, by Going-Up Theorem~\ref{thm:going-up} applied to $\varphi_2$ (\ie to extension-of-scalar $F_2$), we deduce from $\varphi_2(\cat Q')=\cat P'\supseteq\cat P$ the existence of $\cat Q\subseteq \cat Q'$ such that $\varphi_2(\cat Q)=\cat P$. Since $\cat Q$ belongs to $\adhpt{\cat Q'}$, it is also in the $\Spc(C)$ part. Now $\varphi_1(\cat Q)\subseteq\varphi_1(\cat Q')\subseteq\cat P$ and since $\varphi_A\varphi_1=\varphi_A\varphi_2$, the point $\varphi_1(\cat Q)$ belongs to the same $\varphi_A$-fiber as $\cat P$. Since $\cat P$ is minimal in its fiber, we must have $\varphi_1(\cat Q)=\varphi_1(\cat Q')=\cat P$. In short, we have $\cat Q\subseteq\cat Q'$ in $\Spc(C)$ such that $\psi_1(\cat Q)=\psi_1(\cat Q')$. As explained above, we can apply the induction hypothesis to~$\psi_1$, hence deduce $\cat Q=\cat Q'$. Taking the image under $\varphi_2$ finally gives $\cat P=\cat P'$, as wanted.

For~\eqref{it:finite}, let $\cat P_0\in\SpcK$ and consider its fiber $\varphi_A\inv(\cat P_0)\subseteq\Spc(\AK)$. We need to show that $|\varphi_A\inv(\cat P_0)|\leq d$. If it is empty, we are done. Otherwise, pick $\cat P\in \varphi_A\inv(\cat P_0)$. Consider Diagram~\ref{eq:main} again. We are going to define an injection $\varphi_A\inv(\cat P_0)\oursetminus\{\cat P\}\hook \psi_1\inv(\cat P)\subseteq\Spc(C)$. This is sufficient since by induction hypothesis, the fibers of $\psi_1$ have at most~$d-1$ points. For every $\cat P'\in\varphi_A\inv(\cat P_0)$ with $\cat P'\neq\cat P$, there exists by Lemma~\ref{lem:pb} a prime $\cat Q'\in\Spc(A\potimes{2}\MModcat{K})$ such $\varphi_2(\cat Q')=\cat P'$ and $\varphi_1(\cat Q')\subseteq\cat P$. Since $\varphi_A\varphi_1=\varphi_A\varphi_2$, the prime $\varphi_1(\cat Q')$ goes to~$\cat P_0$ under $\varphi_A$ and therefore $\varphi_1(\cat Q')=\cat P$ by Incomparability~\eqref{it:incomp}. Under the decomposition of $\Spc(A\potimes{2})$ as $\Spc(A)\sqcup\Spc(C)$, we cannot have $\cat Q'\in\Spc(A)$ for then the relations $\cat P=\varphi_1(\cat Q')$ and $\cat P'=\varphi_2(\cat Q')$ force $\cat P=\cat P'$ (since $\varphi_1=\id$ and $\varphi_2=\id$ on the $\Spc(A)$ part). Hence $\cat Q'$ belongs to the $\Spc(C)$ part and $\varphi_1(\cat Q')=\cat P$ now reads $\psi_1(\cat Q')=\cat P$. Therefore, we have constructed for every $\cat P'\in\varphi_A\inv(\cat P_0)\oursetminus\{\cat P\}$ a prime $\cat Q'\in\psi_1\inv(\cat P)$ such that $\psi_2(\cat Q')=\cat P'$. The latter relation shows that different $\cat P'$ have different $\cat Q'$. Hence the wanted injection $\varphi_A\inv(\cat P_0)\oursetminus\{\cat P\}\hook\psi_1\inv(\cat P)$.
\end{proof}

\begin{Cor}
\label{cor:w-pb}%
Let $A\in\cat K$ be a tt-ring of finite degree and let $F:\cat K\to \cat L$ be a tt-functor. Let $B=F(A)$ in~$\cat L$. Then the induced commutative square~\eqref{eq:pb}
$$
\vcenter{\xymatrix{
& \SpcK
& \Spc(\cat L) \ar[l]_-{\varphi}
& \cat Q \ar@{}[l]|-{\ni}
\\
\cat P \ar@{}[r]|-{\in}
& \Spc(A\MModcat{K}) \ar[u]^-{\varphi_A}
& \Spc(B\MModcat{L}) \ar[l]_-{\bar \varphi} \ar[u]_-{\varphi_B}
}}
$$
is \emph{weakly cartesian} in the following sense\,: For every $\cat P\in\Spc(\AK)$ and $\cat Q\in\Spc(\cat{L})$ such that $\varphi_A(\cat P)=\varphi(\cat Q)$ there exists $\cat Q'\in \Spc(B\MModcat{L})$ such that $\bar\varphi(\cat Q')=\cat P$ and $\varphi_B(\cat Q')=\cat Q$.
\end{Cor}

\begin{proof}
Choose $\cat Q'$ as in Lemma~\ref{lem:pb}, \ie with $\bar\varphi(\cat Q')\subseteq \cat P$ and $\varphi_B(\cat Q')=\cat Q$. Then $\varphi_A(\bar\varphi(\cat Q'))=\varphi(\varphi_B(\cat Q'))=\varphi(\cat Q)=\varphi_A(\cat P)$. So $\bar\varphi(\cat Q')=\cat P$ by Theorem~\ref{thm:main}\,\eqref{it:incomp}.
\end{proof}

\begin{Rem}
The above square is not cartesian in general. For instance, using $\bbC\otimes_\bbR\bbC\simeq\bbC\times\bbC$, we see that for $\cat K=\Db(\bbR)$ and $A=\bbC$, and thus $\cat L=\Db(\bbC)$, the square has one point in three corners but two points in the bottom-right corner.
\end{Rem}

\begin{Thm}
\label{thm:coeq}%
Let $A$ be a tt-ring of finite degree. Then we have a coequalizer
$$
\xymatrix{\Spc(A\potimes{2}\MModcat{K})\, \ar@<.2em>[r]^-{\varphi_1} \ar@<-.2em>[r]_-{\varphi_2}
& \ \Spc(A\MModcat{K})\, \ar[r]^-{\varphi_A}
& \,\supp(A)
}
$$
where $\varphi_i$ is the map induced by extension-of-scalars $F_i:A\MModcat{K}\too A\potimes{2}\MModcat{K}$ along the two obvious homomorphisms $f_i:A\to A\otimes A$, \ie $f_1=1\otimes \eta$ and $f_2=\eta\otimes 1$.
\end{Thm}

\begin{proof}
Let $W$ be the coequalizer of $\varphi_1$ and~$\varphi_2$. Since $\varphi_A\varphi_1=\varphi_A\varphi_2$, the map $\varphi_A$ induces a continuous map $\bar\varphi:W\to \supp(A)$. By Theorem~\ref{thm:varphi}\,\eqref{it:image}, the map~$\bar\varphi$ is surjective. It is injective by Corollary~\ref{cor:w-pb} for $F=F_A$. (See the explanations in the second half of Remark~\ref{rem:barF} again.) Finally $\bar\varphi$ is a closed map since $\varphi_A:\Spc(\AK)\to \SpcK$ is a closed map by Theorem~\ref{thm:varphi}\,\eqref{it:S}.
\end{proof}

We now want to discuss conditions to replace the above $\supp(A)$ by~$\SpcK$.

\begin{Prop}
\label{prop:solid}%
Let $A$ be an object in~$\cat K$. The following are equivalent\,:
\begin{enumerate}[\rm(i)]
\item
\label{it:solid}%
It has maximal support\,: $\supp(A)=\SpcK$.
\smallbreak
\item
The thick $\otimes$-ideal generated by~$A$ is the whole~$\cat K$; in symbols\,: $\ideal{A}=\cat K$.
\end{enumerate}
If $A$ is a (separable) ring object then the above properties are also equivalent to\,:
\begin{enumerate}[\rm(i)]
\setcounter{enumi}{2}
\item
\label{it:nil-xi}%
In some (hence in every) distinguished triangle $J\oto{\xi}\unit\oto{\eta}A\oto{\zeta}\Sigma J$ over the unit map~$\eta$, the morphism $\xi$ is $\otimes$-nilpotent.
\smallbreak
\item
\label{it:F_A-wf}%
Extension-of-scalars $F_A:\cat K\to \AK$ is \emph{nil-faithful}\,: If a morphism $f$ in $\cat K$ is such that $F_A(f)=0$ (\ie $A\otimes f=0$) then $f$ is $\otimes$-nilpotent.
\smallbreak
\item
\label{it:U_A-gen}%
The image of the forgetful functor $U_A:A\MModcat{K}\to \cat K$ generates~$\cat K$ as a thick $\otimes$-ideal; in symbols\,: $\ideal{U_A(\AK)}=\cat K$.
\end{enumerate}
\end{Prop}

\begin{proof}
(i)$\iff$(ii) is standard, see~\cite[Cor.\,2.5]{Balmer05a}. (ii)$\then$(iii) follows from the observation that $\SET{x\in \cat K}{\xi\otimes x\textrm{ is $\otimes$-nilpotent}}$ is a thick $\otimes$-ideal of~$\cat K$ which contains~$A$ since $\eta\otimes A$ is a split monomorphism. (iii)$\then$(iv)\,: If $f:x\to y$ is such that $A\otimes f=0$ then $(\eta\otimes y)f=0$ and therefore $f$ factors through $\xi\otimes y$ which is $\otimes$-nilpotent. (iv)$\then$(iii) is clear.  (iii)$\then$(ii) is well-known\,: by the Octahedron axiom $\cone(\xi\potimes{n})\in\ideal{\cone(\xi)}=\ideal{A}$ and if $\xi\potimes{n}=0\,:\ J\potimes{n}\to \unit\potimes{n}$ then $\unit\in\ideal{\cone(\xi\potimes{n})}$. To show (ii)$\iff$(v), it suffices to note that $\ideal{U_A(\AK)}=\ideal{A}$ in general. Since $A=U_A(\unit_A)$, only the inclusion $\subseteq$ requires verification. If $(x,\varrho)$ is an $A$-module then $\varrho:A\otimes x\to x$ is split by $\eta\otimes x$, hence $\ideal{A}\ni x=U_A(x,\varrho)$.
\end{proof}

\begin{Def}
\label{def:solid}%
Let us say that an object $A\in \cat K$ is \emph{nil-faithful} if for every morphism $f$ in $\cat K$, we have that $A\otimes f=0$ implies $f\potimes{n}=0$ for some~$n\geq0$. For a tt-ring~$A$ this is equivalent to any of the properties of Proposition~\ref{prop:solid}, see~\eqref{it:F_A-wf}.
\end{Def}

\begin{Cor}
\label{cor:solid}%
Let $F:\cat K\to \cat L$ be a tt-functor. If $A$ is a nil-faithful tt-ring in~$\cat K$ then so is $F(A)$ in~$\cat L$.
\end{Cor}

\begin{proof}
Condition~\eqref{it:nil-xi} in Proposition~\ref{prop:solid} is preserved by tt-functors.
\end{proof}

\begin{Rem}
This Corollary really uses the tensor. For instance, if $A$ only generates $\cat K$ as (thick) triangulated category, it is not necessarily true of~$F(A)$ in~$\cat L$.
\end{Rem}

\begin{Thm}[Descent-up-to-nilpotence]
\label{thm:nil-desc}%
Let $A$ be a \emph{nil-faithful} tt-ring of finite degree. With notation of Theorem~\ref{thm:coeq}, we have a coequalizer of topological spaces\,:
$$
\xymatrix{\Spc(A\potimes{2}\MModcat{K}) \ar@<.2em>[r]^-{\varphi_1} \ar@<-.2em>[r]_-{\varphi_2}
& \Spc(A\MModcat{K}) \ar[r]^-{\varphi_A}
& \SpcK\,.
}
$$
\end{Thm}

\begin{proof}
In Theorem~\ref{thm:coeq}, use that $\supp(A)=\SpcK$ by Proposition~\ref{prop:solid}.
\end{proof}
\goodbreak

%------------------------------------------------------------------------------
\goodbreak
\section{Applications}
\label{se:strat}%
\medbreak
%------------------------------------------------------------------------------

Let $G$ be a finite group and $\Gsets$ the category of finite left $G$-sets. Let $\kk$ be a field of positive characteristic~$p$ dividing~$|G|$, typically just the prime field $\bbF_{\!p}=\bbZ/p$.
\begin{Cons}
\label{cons:AX}%
Let $\PZG$ be the category of those finitely generated left $\bbZ G$-modules which are free as $\bbZ$-modules. It is a tensor category with $\otimes_\bbZ$ and diagonal $G$-action. Consider for every $X\in \Gsets$ the permutation $\bbZ G$-module $\bbZ X$ and, following~\cite{Balmer15}, define on it a $\bbZ$-bilinear multiplication $\mu_X$ by setting $x\cdot x=x$ for all $x\in X$ and $x\cdot x'=0$ when $x\neq x'$. This gives a ring object $\uA(X)=(\bbZ X,\mu_X,\eta_X)$ in the category~$\PZG$ with unit $\eta_X:\bbZ\to \bbZ X$ given by $1\mapsto \sum_{x\in X}x$. For every $G$-map $f:X\to Y$, define $\uA(f):\uA(Y)\to \uA(X)$ by $y\mapsto \sum_{x\in f\inv(y)}x$. For instance, $\uA(X\to *)=\eta_X$. The ring object $\uA(X)$ is separable via $\sigma:\uA(X)\to \uA(X)\otimes \uA(X)$ given by $x\mapsto x\otimes x$. (Note that this map does not come from $\Gsets$.)

Consequently, for every tensor functor $\Phi:\PZG\too \cat K$ to a tt-category $\cat K$, for instance the standard $\PZG\too \DbG$ or $\PZG\too\stab(\kk G)$, the composite functor  $A:=\Phi\circ \uA:\Gsets\op\too\cat K$ provides tt-rings $A(X)$ in~$\cat K$ and ring homomorphisms $A(f):A(Y)\to A(X)$ for every $f:X\to Y$ in~$\Gsets$.
\end{Cons}

\begin{Exa}
For $X=G/H$ an orbit, our tt-ring $A^G_H=\kk(G/H)$ is such an~$A(X)$.
\end{Exa}

\begin{Thm}[Serre's vanishing of Bocksteins, gone tt]
\label{thm:Serre}%
Consider as above $A:\Gsets\op\too\DbG$ given by $X\mapsto A(X)=\kk X$. Let\footnote{In French Picardie, ``abelem" is a clever contraction for ``ab\'elien \'el\'ementaire".} $\Abelem(G)$ be the collection of elementary abelian $p$-subgroups of~$G$ (or only the maximal ones, or representatives up to conjugacy). Then the following tt-ring in $\DbG$
$$
A_{\textrm{elem}}:=\prod_{H\in\Abelem(G)}A(G/H)
$$
is nil-faithful in the sense of Definition~\ref{def:solid} (and Proposition~\ref{prop:solid}).
\end{Thm}

\begin{proof}
Let $\cat K=\DbG$. Since $A(G/H)\otimes-\simeq\Res^G_H$ is faithful when $[\,G:H\,]$ is prime to~$p$, we reduce to $G$ a $p$-group. By induction on~$|G|$ it suffices to show that for every $p$-group $G$ which is not elementary abelian, the tt-ring
$$
A:=\prod_{H\lneqq G}A(G/H)
$$
is nil-faithful in~$\DbG$. By Proposition~\ref{prop:solid}\,\eqref{it:U_A-gen} again, it suffices to prove that $\ideal{A}=\DbG$. Now, Serre's Theorem~\cite[Prop.\,4]{Serre65} precisely says that there are proper subgroups (of index~$p$) $H_1,\ldots ,H_m$ of~$G$ such that the product $\beta(z_1)\cdots \beta(z_m)$ vanishes in $\rmH^{2m}(G,\kk)$, where $\beta(z_i)\in\rmH^2(G,\kk)$ is the Bockstein element associated to $z_i:G\onto G/H_i\simeq\bbZ/p$. As an element of $\Ext^2_{\kk G}(\kk,\kk)\cong\rmH^{2}(G,\kk)$, this Bockstein is given by an exact sequence of the form
\begin{equation}
\label{eq:Bock}%
\beta(z_i)=\,\Big(\,0\to \kk \to \kk(G/H_i) \to \kk(G/H_i) \to \kk\to 0\,\Big)\,.
\end{equation}
If we also denote by $\beta(z_i):\kk\to \kk[2]$ the corresponding element in $\Homcat{K}(\kk,\kk[2])\cong\Ext^2_{\kk G}(\kk,\kk)$ then its cone is the complex $\cdots \to 0\to \kk(G/H_i) \to \kk(G/H_i) \to 0\to \cdots$ appearing above, ``in the middle". In particular, $\cone(\beta(z_i))\in \ideal{\kk(G/H_i)}=\ideal{A(G/H_i)}$. Therefore, an easy application of the octahedron (or just Yoneda splice) shows that the cone of our product $\beta(z_1)\cdots \beta(z_m)=\beta(z_1)\otimes \cdots\otimes \beta(z_m)$ belongs to $\ideal{A(G/H_1), \ldots,A(G/H_m)}\subset \ideal{A}$. Since this product is zero and since $\cone(0:\kk\to \kk[2m])=\kk[1]\oplus\kk[2m]$ we get $\kk\in \ideal{A}$ as wanted.
\end{proof}

\begin{Rem}
I am thankful to Rapha\"el Rouquier for simplifying my earlier proof of Theorem~\ref{thm:Serre}, which involved the following, perhaps interesting, observation. The exact sequence~\eqref{eq:Bock} is the cone of an isomorphism $\psi_i:J_i\isoto J_i^*[-2]$ in $\DbG$, where $J_i=(\cdots0\to \kk\to \kk(G/H_i)\to 0\cdots)$ fits in the distinguished triangle $J_i\oto{\xi_i}\kk\oto{\eta_i} A(G/H_i)\to \Sigma J_i$ as in Proposition~\ref{prop:solid}. A direct computation shows that the Bockstein~$\beta(z_i)$ is equal to the composite $\xi_i[2]\circ\psi_i\inv\circ\xi_i^*$. In other words, our morphism $\xi_i$ appears ``twice" in the corresponding Bockstein.

The analogue of Theorem~\ref{thm:Serre} also holds for $\cat K=\stab(\kk G)$ by Corollary~\ref{cor:solid}.
\end{Rem}

\begin{Def}
\label{def:Or(G)}%
Recall that the \emph{orbit category} $\Or(G)$ of the group~$G$ has objects indexed by subgroups $H\leq G$, thought of as the corresponding orbit, \ie morphisms are given by $\Mor_{\Or(G)}(H,K)=\Mor_{\Gsets}(G/H,G/K)$. We write $\Or(G,\Abelem)$ for the full subcategory of $\Or(G)$ on those $H$ which are elementary abelian $p$-groups.
\end{Def}

\begin{Rem}
\label{rem:blahblah}%
In~\cite{Carlson00}, Carlson already noted that Serre's theorem implies that the support variety $\cV_G=\Proj(\rmH^\sbull(G,\kk))$ is covered by the images of the~$\cV_H$
\begin{equation}
\label{eq:weak-Quillen}%
\cV_G=\bigcup_{H\in\Abelem(G)}\Res^*\cV_H\,.
\end{equation}
This equality is part of Quillen's Stratification Theorem~\eqref{eq:Quillen}, which says that\,:
\begin{equation}
\label{eq:strong-Quillen}%
\cV_G\cong\mathop{\colim}\limits_{H\in\Or(G,\Abelem)}\cV_H\,.
\end{equation}
In terms of cohomology rings, \eqref{eq:weak-Quillen} says that $\rmH^\sbull(G,\kk)\too\prod_H\rmH^\sbull(H,\kk)$ has nilpotent kernel whereas the colimit-version~\eqref{eq:strong-Quillen} is saying moreover that elements in $\lim_H\rmH^\sbull(H,\kk)$ have some $p^r$-power coming from~$\rmH^\sbull(G,\kk)$. See for instance~\cite[Cor.\,II.5.6.4]{Benson98}. Another way to see how the colimit-version~\eqref{eq:strong-Quillen} is more informative than~\eqref{eq:weak-Quillen} is to note that~\eqref{eq:weak-Quillen} does not prevent $\cV_G$ from being a point, for instance, whereas~\eqref{eq:strong-Quillen} allowed Quillen to compute the Krull dimension of~$\cV_G$.

We now want to show that descent-up-to-nilpotence (Theorem~\ref{thm:nil-desc}) yields a generalization of Quillen's Stratification Theorem, in its strong form~\eqref{eq:strong-Quillen}.
\end{Rem}

Quillen's Stratification Theorem~\eqref{eq:strong-Quillen} is the case $\cat K=\stab(\kk G)$ of the following\,:

\begin{Thm}
\label{thm:Quillen}%
Let $\cat K$ be a tt-category and let $\Phi:\Db(\bbF_{\!p}G)\too \cat K$ be a tt-functor. For every subgroup $H\leq G$ consider the tt-ring $A_H=\Phi(\bbF_{\!p}(G/H))$ in~$\cat K$ and define $\cat K(H)=A_H\MModcat{K}$ the corresponding category of modules. Then the collection of tt-functors $F_{A_H}:\cat K=\cat K(G)\too \cat K(H)$ induces a homeomorphism
\begin{equation}
\label{eq:colim}%
\bar\varphi:\mathop{\colim}_{H\in \Or(G,\Abelem)}\Spc(\cat K(H))\isotoo \Spc(\cat K)\,.
\end{equation}
\end{Thm}

\begin{proof}
Following Construction~\ref{cons:AX}, we have a functor $A(-):\Gsets\op\to \cat K$, taking values in tt-rings and homomorphisms and whose restriction to $\Or(G)$ recovers the ring $A(G/H)=\Phi(\bbF_p(G/H))=A_H$ of the statement. Since $A_G=\Phi(\kk)=\unit_{\cat K}$, since $G/G$ is the final $G$-set and since $\Spc(-)$ is contravariant, we indeed have a canonical continuous map $\bar\varphi$ as in~\eqref{eq:colim} such that precomposed with the canonical map~$\Spc(\cat K(H))\to\colim_H\Spc(\cat K(H))$ gives~$\Spc(F_{A_H})$. Consider the tt-ring $A:=\prod_{H\in \Abelem(G)}A_H$, which is the image in~$\cat K$ of the tt-ring $A_{\textrm{elem}}$ of Serre's Theorem~\ref{thm:Serre}. Since $A_\textrm{elem}$ was nil-faithful in $\Dperf(\bbF_{\!p} G)$, Corollary~\ref{cor:solid} implies that $A$ is also nil-faithful in~$\cat K$. By descent-up-to-nilpotence (Theorem~\ref{thm:nil-desc}), we have a coequalizer of spaces
\begin{equation}
\label{eq:colim-A}%
\xymatrix{\Spc(A\potimes{2}\MModcat{K})\ \ar@<.2em>[r]^-{\varphi_1} \ar@<-.2em>[r]_-{\varphi_2}
& \ \Spc(A\MModcat{K}) \ar[r]^-{\varphi_A}
& \SpcK
}
\end{equation}
where the maps $\varphi_i$ are induced by $f_i:A\to A\otimes A$ given by $f_1=1\otimes \eta$ and $f_2=\eta\otimes 1$. Let us identify $A_H\otimes A_K=\Phi(\kk(G/H))\otimes \Phi(\kk(G/K))\cong \Phi(\kk(G/H\times G/K))=A(G/H\times G/K)$, so that the components $A_H\to A_H\otimes A_K$ and $A_K\to A_H\otimes A_K$ of $f_1$ and $f_2:A\to A\otimes A$ are simply the images under~$A:\Gsets\op\to \cat K$ of the two projections $\pr_1:G/H\times G/K\to G/H$ and $\pr_2:G/H\times G/K\to G/K$, respectively.

In $\Gsets$, for any two subgroups $H,K\leq G$, we have a Mackey isomorphism
$$
\coprod_{[g]\in \HGK}\beta_g:\quad \coprod_{[g]\in \HGK}G/(H^g\cap K)\isoto G/H\times G/K
$$
where the map $\beta_g:G/(H^g\cap K)\too G/H\times G/K$ is $[x]\mapsto ([xg\inv],[x])$. As usual, this map is non-canonical since it involves the choice of $g\in G$ for each double class $[g]\in \HGK$. Still, we get an isomorphism (first in~$\Db(\bbF_{\!p} G)$ and then in~$\cat K$)\,:
$$
\prod_{\ourfrac{H,K\in\Abelem(G)}{[g]\in \HGK}}A(\beta_g):\quad A\otimes A=\kern-1em\prod_{H,K\in\Abelem(G)}\kern-2em A_H\otimes A_K\isotoo \prod_{\ourfrac{H,K\in\Abelem(G)}{[g]\in\HGK}}A_{H^g\cap K}\,.
$$
Replacing the decompositions for $A$ and $A\otimes A$ in~\eqref{eq:colim-A}, we obtain a coequalizer
$$
\xymatrix{\coprod\limits_{\ourfrac{H,K\in\Abelem(G)}{[g]\in\HGK}}\Spc(A_{H^g\cap K})
 \ar@<.2em>[r]^-{\psi_1} \ar@<-.2em>[r]_-{\psi_2}
& \coprod\limits_{H\in \Abelem(G)} \Spc(A_H) \ar[rr]^-{\coprod\Displ\varphi_{A_H}}
&& \SpcK
}
$$
where the maps $\psi_1:\Spc(A_{H^g\cap K})\to \Spc(A_H)$ and $\psi_2:\Spc(A_{H^g\cap K})\to \Spc(A_K)$ are induced by the composition of the above maps $f_i:A\xymatrix@C=1em{\ar@<-.2em>[r]\ar@<.2em>[r]&}A\otimes A$ with the maps $A(\beta_g):A_H\otimes A_K\to A_{H^g\cap K}$ on components. Since all maps come through $A:\Gsets\op\to \cat K$, we can compute these compositions in~$\Gsets$ already and get
$$
\xymatrix@C=5em@R=1.5em{
& G/H
\\
G/(H^g\cap K) \ar[r]|-{\,\beta_g\,} \ar[ru]^-{\Displ\pr_1\beta_g=\alpha_g\quad} \ar[rd]_-{\Displ\pr_2\beta_g=\alpha_1\quad}
& G/H\times G/K \ar[u]_-{\pr_1} \ar[d]^-{\pr_2}
\\
& G/K
}
$$
where $\alpha_g:G/L\to G/L'$ denotes $[x]\mapsto [xg\inv]$ whenever ${}^{g\!}L\leq L'$. So, for every $H,K\leq G$ and $g\in G$, we have
\begin{equation}
\label{eq:psi}%
\psi_1{}_{\big|\Spc(A_{H^g\cap K})}=\Spc(A(\alpha_g))
\qquad\text{and}\qquad
\psi_2{}_{\big|\Spc(A_{H^g\cap K})}=\Spc(A(\alpha_1))\,.
\end{equation}
We can now compare our coequalizer with the colimit as follows\,:
$$
\xymatrix@C=1em@R=1em{
\coprod\limits_{\ourfrac{\scriptstyle H,K\in\Abelem(G)}{\scriptstyle [g]\in\HGK}}\kern-1em\Spc(A_{H^g\cap K})
 \ar@<.25em>[rr]^-{\psi_1} \ar@<-.25em>[rr]_-{\psi_2}
&& \coprod\limits_{H\in \Abelem(G)\vphantom{{}^{I^I}}} \kern-1em \Spc(A_H) \ar@{->>}[rr]^-{\Displ\sqcup\,\varphi_{A_H}} \ar@{->>}[rd]^-{\Displ\pi}
&& \SpcK \ar@/^.3em/[ld]^-{\Displ\bar\pi}
\\
&&& \kern-2em \mathop{\colim}\limits_{H\in \Or(G,\Abelem)\vphantom{{}^{I^I}}}\kern-1em \Spc(A_H) \kern-2em \ar@<.2em>[ru]^(.45){\Displ\bar\varphi}
}
$$
where $\pi$ is the quotient map and $\bar\varphi$ is the above continuous map such that~$\bar\varphi\pi=\sqcup_{\scriptscriptstyle H}\varphi_{A_H}$. By~\eqref{eq:psi}, the two maps $\psi_1$ and $\psi_2$ are the images of maps~$\alpha_g$ and $\alpha_1$ in $\Or(G,\Abelem)$ via the functor $H\mapsto \Spc(A_{H})$. Therefore by the colimit property $\pi\circ\psi_1=\pi\circ\psi_2$. Hence, by the coequalizer property, there is a continuous map $\bar\pi:\SpcK\to \colim_{H}\Spc(A_H)$ such that $\bar\pi\circ(\sqcup\varphi_{A_H})=\pi$, as in the above diagram. This $\bar\pi$ is therefore the inverse of~$\bar\varphi$ since $\pi$ and $\sqcup\varphi_{A_H}$ are epimorphisms.
\end{proof}

\begin{Exa}
\label{exa:equiv}%
Let $X$ be a scheme over~$\kk$ on which our finite group $G$ acts. Let $\cat K(G)=\Db_G(\VB_X)$ be the bounded derived category of $G$-equivariant vector bundles on~$X$. Then it receives $\DbG=\Db_G(\Spec(\kk))$ and $A_H\MModcat{\cat K(G)}\cong\cat {}\Db_{H}(\VB_X)$ for every subgroup~$H\leq G$ as can be readily verified from the restriction-coinduction adjunction. Alternatively, see~\cite[\S\,5]{BalmerDellAmbrogioSanders14pp}. Theorem~\ref{thm:Quillen} tells us that the spectrum of~$\Db_G(\VB_X)$ is the colimit of the spectra of $\Db_H(\VB_X)$ over~$H$ in~$\Or(G,\Abelem)$. Furthermore, the Krull dimension of $\Spc(\Db_G(\VB_X))$ is the maximum of the Krull dimensions of $\Spc(\Db_H(\VB_X))$ among the elementary abelian $p$-subgroups~$H\leq G$.
\end{Exa}

%------------------------------------------------------------------------------
\noindent\textbf{Acknowledgements\,:} I would like to thank Serge Bouc, Ivo Dell'Ambrogio, John Greenlees, Alexander Merkurjev, Mike Prest, Rapha\"el Rouquier and Jacques Th\'e\-ve\-naz for helpful discussions. I am also thankful to two anonymous referees for their valuable comments.
%------------------------------------------------------------------------------

%------------------------------------------------------------------------------
\end{document}